\documentclass[12pt,a4]{amsart}
\usepackage{amsmath,amsthm,amsfonts,amssymb,graphicx,float,verbatim,tikz,enumitem,cleveref}
\usepackage[margin=2.5cm]{geometry}

\setenumerate{itemsep=4pt}
\setitemize{itemsep=4pt}

\newtheorem{theorem}{Theorem}
\newtheorem{lem}[theorem]{Lemma}

\newtheorem*{remark}{Remark}

\newcommand{\G}{\mathcal{G}}
\newcommand{\X}{\mathcal{X}}

\tikzstyle{vertex}=[circle, draw, fill=black, inner sep=0pt, minimum width=6pt]
\tikzstyle{pedge}=[draw,-]

\begin{document}

%% Title

\title{On 3-designs from $PGL(2,q)$}
\author[Paul Tricot]{Paul Tricot}
\maketitle

\begin{center}
   \address{Graduate School of Information Sciences, Tohoku University, Japan} 
\end{center}

\section{Introduction}

The group $PGL(2,q)$ acts $3$-transitively on the projective line $GF(q) \cup \{\infty\}$. Thus, an orbit of its action on the $k$-subsets of the projective line is the block set of a $3$-$(q+1,k,\lambda)$ design. In \cite{iwasaki} it has been made clear that the value of the parameter $\lambda$ is directly connected to the order of the stabilizer of a block. The possible orders of stabilizer, and therefore the possible values of $\lambda$ have been computed in \cite{cameron}. Then it becomes interesting to consider a specific $k$-subset of the projective line and see what the parameters of the corresponding design are.

The $3$-designs that arise from the orbit of a block of the form $\langle \theta^r \rangle \cup \{ 0, \infty \}$, where $\theta$ is a primitive element of $GF(q)$, have been considered in \cite{hughes} (without considering the stabilizer).

\begin{theorem}[\cite{hughes}]
    Let $k\geq 2$, $r\geq 1$ such that $q-1=kr$, and $B:=\langle\theta^r\rangle\cup\{0,\infty\}$. Then the $PGL(2,q)$ orbit of $B$ is a simple $3$-$(q+1,k+2,\lambda)$ design where $$\lambda= \begin{cases}
        1 & \text{if $k+1$ divides $q$,}\\
        k+1 & \text{if $k=2$ or $4$ and $k+1$ does not divide $q$,}\\
        (k+2)(k+1)/2 & \text{if $k>4$ and $k+1$ does not divide $q$.}
    \end{cases}$$
\end{theorem}

The orbit of a block of the form $\langle \theta^r \rangle$ or $\langle \theta^r \rangle \cup \{ 0\}$ have been partially studied in \cite{liu} for $r=4x$ or $r=2x$ and with some modulo conditions on $q$ (see Section 3 for details). In this paper, we use similar methods to determine the value of $\lambda$ for designs coming from those two types of block for all values or $r$.

\section{Preliminaries}

% Define design

Let $t,v,k,\lambda$ be positive integers. A $t$-$(v,k,\lambda)$ design is a pair $(X,B)$ where $X$ is a set of size $v$, and $B$ is a collection of $k$-subsets of $X$ called blocks, that verifies: every $t$-subset of $X$ is contained in exactly $\lambda$ blocks of $B$. A double counting argument can be used to prove the well known relation $\lambda {v \choose t} = |B| {k \choose t}$. A design is called simple if $B$ is a set and not a multiset.

% Define PGL, PSL

Let $q$ be a prime power, and $X:=GF(q) \cup \{\infty\}$. The mappings defined on $X$ of the form 
\begin{align*}
    x \mapsto \frac{ax+b}{cx+d}
\end{align*}
where $a,b,c,d \in GF(q)$ and $ad-bc \neq 0$ form the group $PGL(2,q)$. We define the operations involving $\infty$ as usual, and we settle the ambiguous case $\frac{a\infty+b}{c\infty+d} := \frac{a}{c}$. If we add the assumption that $ad-bc$ is a square in $GF(q) \setminus \{0\}$, these mappings form the group $PSL(2,q)$. When $q$ is odd, the order of these groups are $|PSL(2,q)|=\frac{1}{2}q(q^2-1)$ and $|PGL(2,q)|=q(q^2-1)$. When $q$ is even, $PGL(2,2^n)=PSL(2,2^n)$ and they have order $q(q^2-1)$.

The possible subgroups of $PGL(2,q)$ are known.

\begin{theorem}[\cite{cameron, dickson}] \label{thm2}
    Let $q=p^n$ where $p$ is a prime. The subgroups of $PGL(2,q)$ are amongst:
    \begin{itemize}
        \item Cyclic subgroups $C_d$ of order $d$ where $d|q\pm 1$,
        \item Dihedral subgroups $D_{2d}$ of order $2d$ where $d|q\pm 1$,
        \item The alternating group $A_4$,
        \item The alternating and symmetric groups $S_4$ and $A_5$ when $p$ is odd,
        \item $PSL(2,p^m)$ where $m|n$,
        \item $PGL(2,p^m)$ where $m|n$,
        \item The elementary abelian subgroups $Z_p^m$ of order $p^m$ where $m \leq n$,
        \item The semidirect product $Z_p^m \rtimes C_d$, where $m \leq n$, $d|q-1$ and $d|p^m-1$.
    \end{itemize}
\end{theorem}

The possible subgroups of $PSL(2,q)$ have been classified in a similar manner, and we will extract the following results.

\begin{lem}[\cite{dickson}] \label{lem0}
    $PSL(2,q)$ has dihedral subgroups $D_{2d}$ and cyclic subgroups $C_{d}$ only when $d|\frac{q\pm 1}{2}$.
\end{lem}

\begin{lem}[\cite{dickson}] \label{lem8}
    If $p>2$, then $PSL(2,p^n)$ has subgroups $PGL(2,p^m)$ only when $2m|n$.
\end{lem}

Further, the lengths of the orbits of subgroups of $PSL(2,q)$ and $PGL(2,q)$ acting on $X=GF(q) \cup \{\infty\}$ have been well studied in \cite{cameron} and \cite{liu}, respectively. We will write $H \leq G$ when $H$ is a subgroup of $G$.

\begin{lem}[\cite{cameron}] \label{lem1}
    If $H\simeq S_4$ and $H \leq PGL(2,q)$, then $H$ acting on $X$ has orbits of lengths in
    \begin{itemize}
        \item $\{4,6,24\}$ if $q$ is a power of $3$,
        \item $\{6,8,12,24\}$ if $q$ is not a power of $3$.
    \end{itemize}
\end{lem}

\begin{lem}[\cite{cameron}] \label{lem2}
    If $H\simeq A_5$ and $H \leq PGL(2,q)$, then $H$ acting on $X$ has orbits of lengths in $\{10, 12, 20, 30, 60\}$. 
\end{lem}

\begin{lem}[\cite{cameron}] \label{lem3}
    Suppose that $H \leq PGL(2,p^n)$ and $H\simeq PSL(2,p^m)$ or $H\simeq PGL(2,p^m)$ where $m|n$. Then $H$ acting on $X$ has an orbit of length $p^m+1$, at most one orbit of length $p^m(p^m-1)$, and the other orbits are regular.
\end{lem}

\begin{lem}[\cite{cameron}] \label{lem4}
    Suppose that $H \leq PGL(2,q)$ is a dihedral group of order $2d$ where $d|q\pm 1$. Then $H$ acting on $X$ has at most one orbit of length 2, at most two orbits of length $d$, and the other orbits are regular.
\end{lem}

\begin{lem}[\cite{cameron}] \label{lem6}
    Suppose that $H \leq PGL(2,q)$ is a cyclic group of order $d$ where $d|q\pm 1$. Then $H$ acting on $X$ has at most two orbits of length 1, and the other orbits are regular.
\end{lem}

\begin{lem}[\cite{cameron}] \label{lem7}
    Suppose that $H \leq PGL(2,q)$ and $H \simeq Z_p^m \rtimes C_d$ where $m \leq n$, $d|q-1$ and $d|p^m-1$. Then $H$ acting on $X$ has one orbit of length 1, one orbit of length $p^m$, and the other orbits are regular.
\end{lem}

The action of a group $\G$ on $\X$ is said to be $t$-homogeneous if the induced action of $\G$ on $t$-subsets of $\X$ is transitive. If the action of $\G$ on $\X$ is $t$-homogeneous, then for $k \geq t$, any orbit of $\G$ acting on ${\X \choose k}$ form a $t$-$(|\X|,k,\lambda)$ designs for some $\lambda$.

A recurring method (\cite{liu,cameron,balachandran}) to construct $3$-designs is to consider orbits of $PGL(2,q)$ acting on $k$-subsets of $X$ for $k \geq 4$. $PGL(2,q)$ is known to be $3$-homogeneous but not $4$-homogeneous, so the orbits of its action on the $k$-subsets are $3$-designs. Furthermore, when $k \geq 4$ the action on $k$-subsets has several orbits so the resulting $3$-designs are not trivial.

Denote $G:=PGL(2,q)$. If we fix a $k$-subset $B$ of $X$, then the orbit $G(B)$ forms a $3$-$(q+1,k,\lambda)$ design with $\lambda = {k \choose 3} {q+1 \choose 3}^{-1}|G(B)|$. We also make use of the orbit stabilizer theorem: $|G_B||G(B)|=|G|$, so computing the order of the stabilizer $G_B$ is enough to find $\lambda= k(k-1)(k-2)|G_B|^{-1}$. We also notice that since $\lambda$ is an integer, the following condition must be met.

\begin{lem} \label{lem5}
    For any $B \in {X \choose k}$, $|G_B|$ divides $k(k-1)(k-2)$.
\end{lem}

In order to compute the stabilizer of specific $k$-subsets $B$, we will use the property that $B$ is the union of some orbits of its stabilizer in $\G$. This fact coupled with the lemmas above will be used to determine stabilizers of specific $k$-subsets.

There is a connection between $k$-orbits of $PSL(2,q)$ and $k$-orbits of $PGL(2,q)$. When $p$ is odd, $PSL(2,q)$ is a normal subgroup of $PGL(2,q)$ of index 2, and a simple function in $PGL(2,q) \setminus PSL(2,q)$ is $x \mapsto \theta x$ where $\theta$ is a primitive element of $GF(q)$. Therefore it can be proved that if $B$ is a $k$-subset of $X$, and $\Gamma$ is the orbit of $PSL(2,q)$ containing $B$, then $\Gamma\cup\theta\Gamma$ is the orbit of $PGL(2,q)$ containing $B$. For this reason, some authors (\cite{liu,balachandran}) consider $PSL(2,q)$ to study $3$-designs that are orbits of $PGL(2,q)$.

\section{Some $3$-designs}

In this section $q=p^n$ where $p$ is a prime, $G=PGL(2,q)$ and $\theta$ is a primitive element of $GF(q)$. The following results are strengthening of \cite[Theorems 5.1 and 5.5]{liu}. The authors of \cite{liu} considered a block $\langle \theta^4 \rangle$ for $q \equiv 1$ or $13 \pmod {16}$, a block $\langle \theta^{4x} \rangle$ for $q \equiv 1 \pmod{28}$ and $q \equiv 1 \pmod{44}$, and a block $\langle \theta^{2x} \rangle \cup \{0\}$ for $q \equiv 1 \pmod{24}$. Here we find the stabilizer and the value of $\lambda$ for any value of $r$, and without modulo condition on $q$.

\begin{lem}\label{mlem1}
    Let $k\geq 4$, $r\geq 1$ such that $q-1=kr$, and $B:=\langle\theta^r\rangle$. Then
    \begin{enumerate}
        \item if $k-1$ does not divide $q$, then $G_B$ is dihedral of order $2k$,
        \item if $k=p^m+1$ for some $m \leq n$, then $G_B \simeq PGL(2,p^m)$.
    \end{enumerate}
\end{lem}

\begin{proof}
    $G_B$ must be a subgroup from the list of \Cref{thm2}. Moreover, $x \mapsto 1/x$ and $x \mapsto \theta^r x$ stabilize $B$ so $\langle x \mapsto 1/x, x \mapsto \theta^r x \rangle = D_{2k} \leq G_B$. Then $G_B$ can not be any of $A_4$, $C_d$, $Z_p^m$ or $Z_p^m \rtimes C_d$ because they do not have a subgroup $D_{2k}$ where $k\geq 4$.

    Consider the possibility $G_B \simeq S_4$. The largest dihedral subgroup of $S_4$ is $D_8$, so $D_{2k}=D_8$ and $k=|B|=4$. Since $B$ is a union of some orbits of its stabilizer, from \Cref{lem1} we know that $q$ must be a power of 3. Since $A_4 \simeq PGL(2,3)$, this situation corresponds to $G_B \simeq PGL(2,p^m)$ that we will consider later.

    Next we consider the possibility $G_B \simeq A_5$. The largest dihedral subgroup of $A_5$ is $D_{10}$, so since $D_{2k} \leq G_B$, we have $|B| = k \leq 5$. From \Cref{lem2} the orbits of $G_B$ have length at least 10, so $|B|\geq 10$, which gives a contradiction.

    Now suppose that $G_B \simeq PSL(2,p^m)$ where $m|n$. Since $D_{2k} \leq G_B$, using \Cref{lem0} we have $k|\frac{p^m \pm 1}{2}$, so $k \leq \frac{p^m+1}{2}$. Then using the divisibility condition of \Cref{lem5} we have $\frac{1}{2} p^m(p^{2m}-1) \leq k(k-1)(k-2)$. Combining the two inequalities:
    \begin{align*}
        0 &\geq \frac{1}{2}(2k-1)((2k-1)^2-1) - k(k-1)(k-2)\\
        &= \frac{1}{2}(2k-1)(4k^2-4k) - k(k-1)(k-2)\\
        &= k(k-1) \Bigl( 2(2k-1) - (k-2) \Bigr)\\
        &= 3k^2(k-1)
    \end{align*}
    which contradicts the assumption that $k\geq 4$. Thus $G_B$ can not be $PSL(2,p^m)$.
    
    Suppose $G_B \simeq D_{2d}$ where $d|q\pm 1$. Since $D_{2k} \leq G_B$, we have $k|d$ so $k \leq d$. Since $k$ must be the sum of some length of orbits of $G_B$, using \Cref{lem4} we can see that the only possibility is $k=d$, so $G_B \simeq D_{2k}$.
    
    Finally, suppose $G_B \simeq PGL(2,p^m)$ where $m|n$. Similarly, using \Cref{thm2} and \Cref{lem5} we get $k \leq p^m+1$ and $p^m(p^{2m}-1) \leq k(k-1)(k-2)$. The latter inequality implies $p^m+1 \leq k$, so this time the only possibility is $k=p^m+1$.

    Therefore when $k-1$ does not divide $q$, we have $G_B \simeq D_{2k}$. Now we consider the case where $k-1$ does divide $q$, so $k=p^m+1$ for some $m\leq n$. In this case $G_B \simeq D_{2k}$ or $PGL(2,p^m)$.
    
    First we need to show that $m|n$ so that $PGL(2,p^m)$ can be considered. The Euclidean division of $n$ by $m$ can be written $n=xm+y$ (where $0 \leq y \leq m-1$). Then $r(p^m+1)+1=p^n=p^{xm+y}$, so \begin{align*}
        1 &\equiv (p^m)^xp^y \pmod{p^m+1}\\
        &\equiv (-1)^xp^y \pmod{p^m+1}.
    \end{align*}\\
    Thus either $(-1)^xp^y=1$ or $p^m+1 \leq |(-1)^xp^y-1| \leq p^y+1$ which is impossible since $0 \leq y \leq m-1$. We deduce that $y=0$ and $x$ is even, thus $2m|n$.
    
    We will show that the orbit $G(B)$ is the same as the orbit of $B':=GF(p^m) \cup \{\infty\}$. Denote $\beta:=\theta^r$, and consider $f: x \mapsto \frac{x+\beta}{\beta x+1} \in G$. Then $B=\{ 1,\beta,\dots,\beta^{k-1}\}$, $\beta^k=1$ and $\beta^{\frac{k}{2}}=-1$. Thus
    \begin{align*}
        f(\beta^{\frac{k}{2}-1}) &= \frac{\beta^{\frac{k}{2}-1} + \beta}{\beta^{\frac{k}{2}}+1}\\
            &= \infty.
    \end{align*}
    Furthermore, we can show that the remaining points of $B$ are mapped to $GF(p^m)=\{x \in GF(p^n) \mid x^{p^m}=x\}$. Note that $GF(p^m)$ has characteristic $p$, so $(x+y)^{p^m}=x^{p^m}+y^{p^m}$ for any $x,y \in GF(p^m)$. Then, for $\alpha \in \{0,\dots,k-1\} \setminus \{\frac{k}{2}-1\}$,
    \begin{align*}
        f(\beta^\alpha)^{p^m} &= \Bigl( \frac{\beta^\alpha + \beta}{\beta^{\alpha+1}+1} \Bigr)^{p^m}\\
            &= \frac{\beta^{\alpha p^m}+\beta^{p^m}}{\beta^{(\alpha+1)p^m}+1}\\
            &= \frac{\beta^{\alpha (k-1)}+\beta^{k-1}}{\beta^{(\alpha+1)(k-1)}+1}\\
            &= \frac{\beta^{\alpha (k-1)+\alpha+1}+\beta^{k-1+\alpha+1}}{\beta^{(\alpha+1)(k-1)+\alpha+1}+\beta^{\alpha+1}}\\
            &= \frac{\beta+\beta^\alpha}{1+\beta^{\alpha+1}}\\
            &= f(\beta^\alpha).
    \end{align*}
    This shows that $f(B)=B'$, and therefore $G(B)=G(B')$. Using the orbit-stabilizer theorem, this implies that $|G_B|=|G_{B'}|$.

    Since $GF(p^m)$ is a subfield of $GF(p^n)$, we can extend the action of $PGL(2,p^m)$ to $X=GF(p^n) \cup \{\infty\}$. Then $PGL(2,p^m) \leq PGL(2,p^n)=G$ and $PGL(2,p^m)$ stabilize $B'$ so $PGL(2,p^m) \leq G_{B'}$. Thus looking at the list of \Cref{thm2} and \Cref{lem8}, $G_{B'}$ must be $PGL(2,p^{m'})$ with $m|m'|n$ or $PSL(2,p^{m'})$ with $2m|m'|n$ ($p>2$). Then since $B'$ is a union of some orbits of its stabilizer, \Cref{lem3} insure that we have in fact $PGL(2,p^m)=G_{B'}$. Finally, since $|G_B|=|G_{B'}|$ and the possibilities for $G_B$ were reduced to $D_{2k}$ or $PGL(2,p^m)$, we can conclude that $G_B \simeq PGL(2,p^m)$.
\end{proof}

\begin{theorem}
    Let $k\geq 4$, $r\geq 1$ such that $q-1=kr$, and $B:=\langle\theta^r\rangle$. Then the $PGL(2,q)$-orbit of $B$ is a simple $3$-$(q+1,k,\lambda)$ design where $$\lambda= \begin{cases}
        (k-1)(k-2)/2 & \text{if $k-1$ does not divide $q$,}\\
        1 & \text{if $k-1$ divides $q$.}
    \end{cases}$$
\end{theorem}
\begin{proof}
    The orbit of $B$ is a $3$-$(q+1,k,\lambda)$ design where $\lambda = k(k-1)(k-2)|G_B|^{-1}$, and the order of $G_B$ is given by \Cref{mlem1}.
\end{proof}

\begin{lem}\label{mlem3}
    Let $k\geq 3$, $r\geq 1$ such that $q-1=kr$, and $B:=\langle\theta^r\rangle\cup\{0\}$. Then
    \begin{enumerate}
        \item \label{case 1} if $k=3$, then $G_B\simeq A_4$,
        \item \label{case 2} if $k\geq 4$ and $k+1$ does not divide $q$, then $G_B$ is cyclic of order $k$,
        \item \label{case 3} if $k\geq 4$ and $k+1=p^m$ for some $m\leq n$, then $G_B$ is a semidirect product $Z_p^m \rtimes C_{p^m-1}$.
    \end{enumerate}
\end{lem}
\begin{proof}
    We proceed again by elimination from the list of \Cref{thm2}. $x \mapsto \theta^r x$ stabilizes $B$ so $\langle x \mapsto \theta^r x \rangle = C_k \leq G_B$. Then $G_B$ can not be $Z_p^m$ because it does not have a subgroup $C_k$.

    Suppose $G_B \simeq S_4$. Then $C_k$ must be $C_3$ or $C_4$, so $k=3$ or $4$ and $|B|=k+1=4$ or $5$. Since $B$ is a union of some orbits of $G_B$, \Cref{lem1} gives a contradiction (Note that $k=3$ implies that $q$ is not a power of $3$).

    Suppose $G_B \simeq A_5$. Then $C_k$ must be $C_3$ or $C_5$, so $k=3$ or $5$ and $|B|=4$ or $6$. This time \Cref{lem2} gives a contradiction.

    Suppose $G_B \simeq D_{2d}$. Then $C_k \leq D_{2d}$ so $k$ must divide $d$, and \Cref{lem4} gives a contradiction.

    Suppose $G_B \simeq PSL(2,p^m)$ for some $p^m|q$. Then from \Cref{lem5} (with this time $|B|=k+1$), $|G_B|$ divides $(k+1)k(k-1)$, so $\frac{1}{2}p^m(p^{2m}-1) \leq k(k^2-1)$. On the other hand, since $C_k \leq G_B$, from \Cref{lem0} we have $k|\frac{p^m\pm 1}{2}$ so $k \leq \frac{p^m+1}{2}$. Combining the two inequalities:
    \begin{align*}
        0 &\geq \frac{1}{2}(2k-1)((2k-1)^2-1) - (k+1)k(k-1)\\
        &= \frac{1}{2}(2k-1)(4k^2-4k) - (k+1)k(k-1)\\
        &= k(k-1) \Bigl( 2(2k-1) - (k+1) \Bigr)\\
        &= 3k(k-1)^2
    \end{align*}
    which contradicts the assumption that $k\geq 3$.
    
    Suppose $G_B \simeq PGL(2,p^m)$ for some $p^m|q$. Similarly, using \Cref{thm2} and \Cref{lem5} we get $k | p^m\pm 1$ and $p^m(p^{2m}-1) \leq k(k^2-1)$ so $p^m \leq k$. This time the only possibility is $k=p^m+1$, so $|B|=p^m+2$. Then \Cref{lem3} gives a contradiction.

    Suppose $G_B \simeq A_4$. Then $C_k$ must be $C_3$, so $k=3$.

    Suppose $G_B \simeq C_d$. Since $C_k \leq C_d$, we have $k|d$. Then looking at the orbit lengths given in \Cref{lem6}, the only possibility is $k=d$ so $G_B \simeq C_k$.
    
    Suppose $G_B \simeq Z_p^m \rtimes C_{d}$ where $m \leq n$, $d|p^n-1$ and $d|p^m-1$. Since $C_k \leq Z_p^m \rtimes C_{d}$, we have $k|dp^m$. As $k$ does not divide $q$ and thus not $p^m$, we have $k|d|p^m-1$, so $k \leq p^m-1$. Then since $B$ is a union of orbits of $G_B$ and $|B|=k+1$, from \Cref{lem7} the only possibility is $k=d=p^m-1$. Thus $G_B \simeq Z_p^m \rtimes C_{p^m-1}$, $k \neq 3$ and $k+1$ must divide $q$.
    
    Therefore in the case (\ref{case 2}) where $k\geq 4$ and $k+1$ does not divide $q$, we have $G_B \simeq C_k$. Now we consider the case (\ref{case 1}) where $k=3$. In this case, $G_B \simeq C_3$ or $A_4$, where $C_3=\langle x \mapsto \theta^r x \rangle$. Thus in order to show that $G_B$ is $A_4$, it is enough to find an element of $G_B \setminus C_3$. Consider $f:x \mapsto \frac{x-1}{(\theta^{2r}+\theta^r-1)x-1}$. Then $f$ restricted to $B$ is the permutation $(0,1)(\theta^r,\theta^{2r})$, so $f \in G_B \setminus C_3$.
    
    The last case to consider is (\ref{case 3}) where $k\geq 4$ and $k+1$ does divide $q$, so $k+1=p^m$ for some $m\leq n$. In this case, $G_B \simeq C_k$ or $Z_p^m \rtimes C_{p^m-1}$.
    
    First we show that $m|n$. The Euclidean division of $n$ by $m$ can be written $n=xm+y$ (where $0 \leq y \leq m-1$). Then $r(p^m-1)+1=p^n=p^{xm+y}$, so \begin{align*}
        1 &\equiv (p^m)^xp^y \pmod{p^m-1}\\
        &\equiv p^y \pmod{p^m-1}.
    \end{align*}
    Thus either $p^y=1$ or $p^m-1 \leq p^y-1$ which is impossible since $0 \leq y \leq m-1$. We deduce that $y=0$, so $m|n$.

    Then, $B$ is the subfield $GF(p^m)$ of $GF(q)$. Thus the stabilizer of $B$ in $G':=PGL(2,p^m)$ is the same as the stabilizer of infinity, 
    \begin{align*}
        G'_{B} &= G'_{\infty}\\
            &= \{x\mapsto ax+b \mid a \in GF(p^m) \setminus \{0\}, b \in GF(p^m)\}\\
            &\simeq Z_p^m \rtimes C_{p^m-1}.
    \end{align*}
    Since $G'_B \leq G_B$ and $G_B$ was reduced to the two possibilities $C_{p^m-1}$ or $Z_p^m \rtimes C_{p^m-1}$, we can conclude that $G_B\simeq Z_p^m \rtimes C_{p^m-1}$.

\end{proof}

\begin{theorem}\label{thm4}
    Let $k\geq 3$, $r\geq 1$ such that $q-1=kr$, and $B:=\langle\theta^r\rangle\cup\{0\}$. Then the $PGL(2,q)$-orbit of $B$ is a simple $3$-$(q+1,k+1,\lambda)$ design where $$\lambda=\begin{cases}
        2 & \text{if $k=3$,}\\
        (k+1)(k-1) & \text{if $k\geq 4$ and $k+1$ does not divide $q$,}\\
        k-1 & \text{if $k\geq 4$ and $k+1$ divides $q$.}
    \end{cases}$$
\end{theorem}
\begin{proof}
    The orbit of $B$ is a $3$-$(q+1,k+1,\lambda)$ design where $\lambda = (k+1)k(k-1)|G_B|^{-1}$, and the order of $G_B$ is given by \Cref{mlem3}.
\end{proof}

\begin{remark}
    The $3$-designs with block size $4$ that are orbits of $PGL(2,2^n)$ are studied in \cite{keranen2}. The authors show that the orbit of $\{0,1,\theta^r,\infty\}$ where $\theta^r$ is a cube root of unity, is a design with $\lambda=2$. This is in fact the same design as the one in \Cref{thm4} from a block $\{0,1,\theta^r,\theta^{2r}\}$. Indeed, $f: x \mapsto \frac{1}{\theta^{2r}x+\theta^r}$ maps one block to the other.
\end{remark}

\section{Acknowledgement}

I am grateful to Prof. Akihiro Munemasa for his continuous guidance through this study.\\

This work was supported by JST SPRING (Grant Number JPMJSP2114).

\bibliographystyle{plain}
\bibliography{dpgl}

\begin{thebibliography}{1}

\bibitem{balachandran}
Niranjan Balachandran and Dijen Ray-Chaudhuri.
\newblock Simple 3-designs and {${\rm PSL}(2,q)$} with {$q\equiv1\pmod4$}.
\newblock {\em Des. Codes Cryptogr.}, 44(1-3):263--274, 2007.

\bibitem{cameron}
P.~J. Cameron, G.~R. Omidi, and B.~Tayfeh-Rezaie.
\newblock 3-designs from {${\rm PGL}(2,q)$}.
\newblock {\em Electron. J. Combin.}, 13(1):Research Paper 50, 11, 2006.

\bibitem{dickson}
Leonard~Eugene Dickson.
\newblock {\em Linear groups: {W}ith an exposition of the {G}alois field theory}.
\newblock Dover Publications, Inc., New York, 1958.
\newblock With an introduction by W. Magnus.

\bibitem{hughes}
D.~R. Hughes.
\newblock On {$t$}-designs and groups.
\newblock {\em Amer. J. Math.}, 87:761--778, 1965.

\bibitem{iwasaki}
Shiro Iwasaki and Thomas Meixner.
\newblock A remark on the action of {${\rm PGL}(2,q)$} and {${\rm PSL}(2,q)$} on the projective line.
\newblock {\em Hokkaido Math. J.}, 26(1):203--209, 1997.

\bibitem{keranen2}
M.~S. Keranen and D.~L. Kreher.
\newblock 3-designs of {${\rm PSL}(2,2^n)$} with block sizes 4 and 5.
\newblock {\em J. Combin. Des.}, 12(2):103--111, 2004.

\bibitem{liu}
WeiJun Liu, JianXiong Tang, and YiXiang Wu.
\newblock Some new 3-designs from {$PSL(2,q)$} with {$q\equiv 1\pmod4$}.
\newblock {\em Sci. China Math.}, 55(9):1901--1911, 2012.

\end{thebibliography}

\end{document}